\documentclass{amsart}[12pt]
\usepackage{mathrsfs,amsmath,amssymb,latexsym}
\usepackage{amsfonts,amsmath,amsthm}
\usepackage{fullpage}
\usepackage{enumerate}
\newtheorem{thrm}{Theorem}[section]
\newtheorem{lem}[thrm]{Lemma}
\newtheorem{pro}[thrm]{Proposition}

\theoremstyle{definition}

\newtheorem{remark}[thrm]{Remark}
\author[N. Demni]{Nizar Demni}
\address{IRMAR, Universit\'e de Rennes 1\\ Campus de
Beaulieu\\ 35042 Rennes cedex\\ France}
\email{nizar.demni@univ-rennes1.fr}
\date{\today}
\title{Lagrange inversion formula, Laguerre polynomials and the free unitary Brownian motion} 
\begin{document}
\maketitle
\begin{abstract}
This paper is devoted to the computations of some relevant quantities associated with the free unitary Brownian motion. Using the Lagrange inversion formula, we first derive an explicit expression for its alternating star cumulants which have even lengths and relate them to those having odd lengths by means of a summation formula for the free cumulants with product as entries. Next, we use again Lagrange formula together with a generating series for Laguerre polynomials in order to compute the Taylor coefficients of the reciprocal of the $R$-transform of the free Jacobi process associated with a single projection of rank $1/2$, and those of its $S$-transform as well. This generating series lead also to the Taylor expansions of the Schur function of the spectral distribution of the free unitary Brownian motion and of its first iterate. 
\end{abstract}

\section{Introduction}
Free probability theory has its origin in harmonic analysis on free groups and the formulation of its basic concepts is due to D.V. Voiculescu in the eightees. In this framework, a $\star$-noncommutative probability space is a unital noncommutative von Neumann algebra $\mathscr{A}$ equipped with a faithful finite trace $\phi$ and an involution $\star$, and the freeness in Voiculescu's sense replaces the classical independence of commutative random variables. By the spectral Theorem, freeness in Voiculescu's sense leads to highly noncommutative convolutions of probability distributions by adding or multiplying self-adjoint or unitary random variables. Moreover,  analytic transforms analogous to the Fourier transform and linearizing these free convolutions exist and allow to compute the resulting probability distributions. Other probabilistic concepts such as infinite-divisibility, L\'evy processes, L\'evy-Khintchine representations and cumulants have also their free counterparts. For instance, the latter can be defined either analytically as the Taylor coefficients of the so-called $R$-transform which linearizes the free additive convolution or in a combinatorial fashion through the lattice of noncrossing partitions and its M\"obius function. We refer the reader to the monographs \cite{Hia-Pet}, \cite{NS}, \cite{Voi} and references therein for a good and large account on free probability theory and on its operator-algebraic and combinatorial aspects. 

In this paper, we are interested in solving three problems related to the so-called free unitary Brownian motion. Introduced in \cite{Bia}, this is a family of unitary operators $(u_t)_{t \geq 0}$ starting at the unit ${\bf 1}$ of $\mathcal{A}$ whose spectral distributions $(\nu_t)_{t \geq 0}$ form a semigroup with respect to the free multiplicative convolution of probability distribution on the unit circle $\mathbb{T}$. Moreover, $\nu_t$ is invariant under complex conjugation and its positive moments are given at any time $t > 0$ by (\cite{Bia}): 
\begin{equation}\label{mom}
\phi(u_t^k) = \int_{\mathbb{T}} z^k\nu_t(dz) = \frac{e^{-kt/2}}{k}L_{k-1}^{(1)}(kt), \quad k \geq 1,
\end{equation}
where $L_{k-1}^{(1)}$ is the $(k-1)$-th Laguerre polynomial of index $1$ (\cite{AAR}). In particular, $\nu_t$ converges weakly as $t \rightarrow \infty$ to the spectral distribution $\nu_{\infty}$ of a Haar unitary operator $u_{\infty}$ (\cite{Hia-Pet}, \cite{NS}). Besides, the $R$-transform of $\nu_t$ is expressed through the Lambert function so that the sequence of its free cumulants is given by (\cite{DGPN}, \cite{Levy}):
\begin{equation*}
e^{-nt/2}\frac{(-nt)^{n-1}}{n!}, \quad n \geq 1. 
\end{equation*}
More generally, the star cumulants (see below) of the free unitary Brownian motion were studied in \cite{DGPN} and their explicit expressions are only available in some particular cases. Of special interest are those corresponding to alternating words $(u_t, u_t^{\star}, u_t, u_t^{\star}, \dots)$ and having even lengths since they converge as $t \rightarrow \infty$ to the only non zero star cumulants of $u_{\infty}$ (\cite{NS}). The star cumulants of odd lengths are relevant as well since they encode the infinitesimal structure of $u_t$ near $t = \infty$ defined and completely determined in \cite{DGPN}. In that paper, it was shown that the generating function of the alternating star cumulants having even lengths satisfies a non linear partial differential equation (hereafter pde) which was then solved using the method of characteristics. The first problem we deal with here is concerned with the derivation of an explicit formula for these free cumulants and amounts to the computation of the Taylor coefficients of the local inverse around $z=1$ of the map encoding the characteristics of the pde. 

Another family of positive bounded operators -the free Jacobi process $(J_t)_{t \geq 0}$-is closely related to the free unitary Brownian motion. Actually, let $\{P, Q\} \in \mathscr{A}$ be two orthogonal projections which are free with 
$\{u_t, u_t^{\star}\}_{t \geq 0}$ (we assume without loss of generality that $\mathscr{A}$ is large enough to contain both free families of operators). Then, the free Jacobi process associated with $\{P,Q\}$ is defined by (\cite{Dem})
\begin{equation*}
J_t := Pu_tQu_t^{\star}P, \quad t \geq 0,
\end{equation*}
and is considered as an element of the compressed algebra $(P\mathscr{A}P, \phi/\phi(P))$. When $P=Q$ and $\phi(P) = 1/2$, it was proved in \cite{DHH} that the spectral distribution of $J_t$, say $\mu_t$, coincides with that of the self-adjoint random variable 
\begin{equation}\label{DescJac}
\frac{u_{2t} + u_{2t}^{\star} + 2{\bf 1}}{4}
\end{equation}
in $(\mathscr{A}, \phi)$. In particular, the multi-linearity of the free cumulant functional (\cite{NS}) and the freeness of ${\bf 1}$ with $\{u_{2t}, u_{2t}^{\star}\}$ show that any free cumulant of $J_t$ is the sum of star cumulants of $u_{2t}$ of different types. Since the latter are only known in few cases, it is most likely better to derive the free cumulants of $J_t$ by inverting its moment generating function already obtained in \cite{DHH}. In the same spirit, it is natural to seek an explicit expression for the $S$-transform of $\mu_t$. 

Another problem related to the free unitary Brownian motion which we tackle here is the determination of the Verblunsky coefficients of $\nu_t$, known also as the Schur parameters. The knowledge of these complex numbers would imply important properties enjoyed by $\nu_t$ as illustrated for instance by the strong Szeg\"o Theorem (\cite{Sim1}). They are defined as the coefficients of the continued fraction expansion of its Schur function or equivalently, by Geroniums Theorem (\cite{Sim1}), as those of the recurrence relation satisfied by the orthogonal polynomials with respect to $\nu_t$. Equivalently, the Schur algorithm shows that the Verblunsky coefficients may be realized as the constant terms of the Schur iterates. Moreover, they can be connected via the inverse Geronimus relations with the Jacobi-Szeg\"o parameters of the image of $\nu_t$ under the Szeg\"o map (\cite{Sim2}) which, up to affine transformations, coincide with the Jacobi-Szeg\"o parameters of $\mu_{t/2}$. 

In this paper, we solve the first problem related to the star cumulants of $u_t$ and make major steps toward the solutions to the two remaining ones. More precisely, we use Lagrange inversion formula in order to write down the expression of the alternating star cumulants having even lengths of $u_{t}$. Appealing to a summation formula due to Krawczyk and Speicher for the free cumulants with product as entries, we relate the star cumulants having odd lengths to those having even lengths opening therefore the way to compute inductively the former from the latter. Afterwards, we use again Lagrange formula together with a certain generating series for Laguerre polynomials in order to derive the Taylor expansion of the reciprocal of the $R$-transform of $\mu_t$. Similar computations lead also to the Taylor expansions of the $S$-transform of $\mu_t$, of the Schur function of $\nu_t$ and of its first iterate. 

For sake of completeness, we recall in the following section some needed facts from free probability theory as well as the definitions of various special functions occuring in the remainder of the paper. 

\section{Reminder: free probability theory and special functions} 
\subsection{Free probability theory}
Let $(\mathscr{A}, \phi)$ be a $\star$-noncommutative probability space and let $a \in \mathscr{A}$. Then, the $n$-th moment of $a$ is the complex number $\phi(a^n), n \geq 1$ and if $a$ is a (bounded) self-adjoint or a unitary operator, then $\phi(a^n)$ is the $n$-th moment of the spectral distribution of $a$ which is supported in a compact set of the real line $\mathbb{R}$ or in $\mathbb{T}$ respectively (\cite{NS}, p.43-44). Now, denote $NC(n)$ the set of noncrossing partitions and consider the reverse refinement order $\leq$: for $\pi, \rho \in NC(n)$, $\pi \leq \rho$ if and only if that every block of $\rho$ is a union of blocks of $\pi$. This is a partial order and the set $(NC(n), \leq )$ turns out to be a {\em lattice}, that is, every $\pi , \rho \in NC(n)$ 
have a smallest common upper bound $\pi \vee \rho$ and a greatest common lower bound $\pi \wedge \rho$ (\cite{NS}, p144). Moreover, the minimal and maximal elements of $(NC(n), \leq)$ are $0_n$ (the partition of $\{ 1, \ldots , n \}$ into $n$ blocks) and $1_n$ (the partition of $\{ 1, \ldots , n \}$ into one block), and its M\"obius function defined on $\{(\pi,\rho) \in NC(n), \pi \leq \rho\}$ is denoted `Mob' (\cite{NS}, Lecture X). 

Given a $n$-tuple $\{a_1, \dots, a_n\} \in \mathscr{A}$, their joint $n$-th cumulant is defined by (\cite{NS}, p.175-176) 
\begin{equation*}
\kappa_n(a_1,\dots,a_n) = \sum_{\pi \in NC(n)} \textrm{Mob}(\pi, 1_n)\prod_{V \in \pi}\phi(a_1,\dots, a_n)|V)
\end{equation*}
where, for a block $V$ of $\pi$,
\begin{equation*}
\phi(a_1,\dots, a_n)|V) := \phi(a_{i_1}\cdots a_{i_{|V|}}), \quad V = \{i_1 < i_2 \dots < i_{|V|}\} \in \pi.
\end{equation*}
In particular, $\kappa_n$ is a multilinear functional and if $a_1, \dots, a_n \in \{u_t, u_t^{\star}\}$, then we get the {\bf star cumulant} of $u_t$ of length $n$. Besides, the {\bf free cumulants} of $a$ is the sequence $\kappa_n(a), n \geq 1$ defined by: 
\begin{equation*}
\kappa_n(a) := \kappa_n(a,\dots,a) = \sum_{\pi \in NC(n)} \textrm{Mob}(\pi, 1_n)\prod_{V \in \pi}\phi(a^{|V|}),
\end{equation*}
where for a block $V \in \pi$, $|V|$ is its cardinality.  

We will also make use of the following result due to Krawczyk and Speicher (\cite{NS}, pp.178-181) which gives a structured summation 
formula for free cumulants with products as entries (\cite{NS}, p.180). More precisely, let $\sigma = \{ J_1, \ldots , J_k \} \in NC(n)$ be a partition where every block is an interval:
$J_1 = \{ 1, \ldots , j_1 \}, J_2 = \{ j_1 + 1, \ldots , j_2 \}, \ldots , J_k = \{ j_{k-1} +1, \ldots , j_k \}$ for some $1 \leq j_1 < j_2 < \cdots < j_k = n$.  Then for every $a_1, \ldots , a_n \in \mathscr{A}$, one has
\begin{equation}\label{PE}
\kappa_k \bigl( a_1 \cdots a_{j_1}, \ a_{j_1 +1} \cdots a_{j_2}, \, \ldots , \, a_{j_{k-1}+ 1} \cdots a_{j_k} \bigr) = \sum_{\substack{\pi \in NC(n)   \\ \pi \vee \sigma = 1_n}}
 \prod_{V \in \pi} \kappa_{|V|}\bigl( \, (a_1, \ldots , a_n) \mid V \, \bigr)
\end{equation}
where 
\begin{equation*}
\kappa_{|V|}\left(a_1, \ldots , a_n) \mid V \right) :=  \kappa_{|V|}(a_{i_1}, \dots, a_{i_{|V|}}), \quad V = \{i_1 < i_2 \dots < i_{|V|}\} \in \pi.
\end{equation*}
At the analytic side, recall that the $R$-transform of $a$ is the following free cumulant generating function:
\begin{equation*}
R(z) := \sum_{n \geq 1}\kappa_n(a) z^n.
\end{equation*}
Since $a$ is a bounded operator, then this series converges absolutely in a neighborhood of the origin. Moreover, it is related to the moment generating function 
\begin{equation*}
M(z) := 1+\sum_{n \geq 1}\phi(a^n)z^n 
\end{equation*}
by the following functional equation (\cite{NS}, p.269): 
\begin{equation}\label{FuncEq}
w(z) = z[1+R(w(z))]
\end{equation}
where $w(z) := zM(z)$. In other words, the map
\begin{equation*}
F : z \mapsto \frac{z}{1+R(z)}
\end{equation*}
is the compositional inverse of $w$ near $z=0$. Finally, assume further that $\phi(a) \neq 0$. Then the $S$-transform of $a$ is defined by (\cite{Vo})
\begin{equation}\label{FuncEq1}
zS(z) = (1+z)(M-1)^{-1}(z),
\end{equation}
and satisfies the inverse relation $R^{-1}(z) = zS(z)$.  

\subsection{Special functions}
Apart from the last two facts recalled at the end of this paragraph, the remaning ones are standard and may be found in the books \cite{AS} and \cite{AAR}. We start with the Gamma function 
\begin{equation*}
\Gamma(x) = \int_0^{\infty} e^{-u}u^{x-1} du, \quad x > 0, 
\end{equation*}
and the Pochhammer symbol 
\begin{equation*}
(a)_k = (a+k-1)\dots(a+1)a, \quad a \in \mathbb{R}, \, k \in \mathbb{N}, 
\end{equation*}
with the convention $(0)_k = \delta_{k0}$. The latter may be written as  
\begin{equation*}
(a)_k = \frac{\Gamma(a+k)}{\Gamma(a)}
\end{equation*} 
when $a > 0$, while
\begin{equation}\label{I1}
\frac{(-n)_k}{k!} = (-1)^k \binom{n}{k}
\end{equation}
if $k \leq n$ and vanishes otherwise. 
Next comes the generalized hypergeometric function defined by the series 
\begin{equation*}
{}_pF_q((a_i, 1 \leq i \leq p), (b_j, 1 \leq j \leq q); z) = \sum_{m \geq 0}\frac{\prod_{i=1}^p(a_i)_m}{\prod_{j=1}^q(b_j)_m}\frac{z^m}{m!}
\end{equation*}
where an empty product equals one and the parameters $(a_i, 1 \leq i \leq p)$ are reals while $(b_j, 1 \leq j \leq q) \in \mathbb{R} \setminus -\mathbb{N}$. 
With regard to \eqref{I1}, this series terminates when at least $a_i = -n \in - \mathbb{N}$ for some $1 \leq i \leq p$, therefore reduces in this case to a polynomial of degree $n$. In particular, the Charlier polynomials are defined by 
\begin{equation*}
C_n(x,a) := {}_2F_0\left(-n, -x; -\frac{1}{a}\right), \quad a \in \mathbb{R} \setminus \{0\}, x \in \mathbb{R}.   
\end{equation*}
When $x \in \mathbb{Z}$ is an integer, a generating function of these polynomials is given by
\begin{equation}\label{GFC}
\sum_{n \geq 0}C_n(x,a)\frac{(-au)^n}{n!} = e^{-au}\left(1+u\right)^x, \quad |u| < 1. 
\end{equation}
Moreover, the $n$-th Laguerre polynomial with index $\alpha \in \mathbb{R}$ is defined by 
\begin{equation}\label{DefL}
L_n^{(\alpha)}(z) := \frac{1}{n!}\sum_{j=0}^n\frac{(-n)_j}{j!}(\alpha+j+1)_{n-j}z^j, 
\end{equation}
and is related to the $n$-th Charlier polynomial via:
\begin{equation}\label{CL}
\frac{(-a)^n}{n!}C_n(x,a) = L_n^{(x-n)}(a). 
\end{equation}
It also obeys the following differentiation rule: 
\begin{equation}\label{DiffRule}
\frac{d^m}{dx^m}L_{k}^{(\alpha)}(x) = (-1)^mL_{k-m}^{(m+\alpha)}(x).
\end{equation}
We close this paragraph with the statement of Lagrange inversion formula which plays a key role in our subsequent computations as well as an instance of Brown's Theorem (see e.g. \cite{Man-Sri}, p.356-357).
\begin{itemize}
\item Let $z_0 \in \mathbb{C}$ and let $g$ be a holomorphic function in a neighborhood of $z_0$ with $g'(z_0) \neq 0$. Then, $g$ is locally invertible and its inverse function $g^{-1}$ admits the following expansion near $w_0 := g(z_0)$: 
\begin{equation*}
g^{-1}(w) = g^{-1}(w_0) + \sum_{n \geq 1} \frac{(w-w_0)^n}{n!} \left\{\partial_z^{n-1}\left[\frac{z-z_0}{g(z) - g(z_0)}\right]^n\right\}_{|z=z_0}.
\end{equation*}
\item Given a sequence $(r_n)_{n \geq 0}$, define the new one $(p_n)_{n \geq 0}$ by 
\begin{equation*}
p_n := \sum_{k=0}^n \binom{2n}{n-k} r_n. 
\end{equation*}
Then 
\begin{equation*}
\sum_{n \geq 0} p_n w^n = \frac{1+z}{1-z}\sum_{n \geq 0} r_n z^n 
\end{equation*}
whenever both series converge, where $w = z/(1+z)^2$. Equivalently, if \footnote{We take the principal determination of the square root.}
\begin{equation*}
\alpha(w) := \frac{1-\sqrt{1-w}}{1+\sqrt{1-w}}, \quad w \in \mathbb{C} \setminus [1,\infty[,
\end{equation*}
then $\alpha$ is invertible with inverse given by 
\begin{equation*}
w = \alpha^{-1}(z) = \frac{4z}{(1+z)^2}, \quad |z| < 1,
\end{equation*}
so that 
\begin{equation*}
\sum_{n \geq 0} p_n \frac{w^n}{4^n} = \frac{1+\alpha(w)}{1-\alpha(w)}\sum_{n \geq 0} r_n [\alpha(w)]^n = \frac{1}{\sqrt{1-w}}\sum_{n \geq 0} r_n [\alpha(w)]^n. 
\end{equation*}
\end{itemize}

\section{Alternating Star cumulants} 
For $n \geq 1$, let 
\begin{equation*}
g_n(t) := \kappa_{2n}(\underbrace{u_t, u_t^{\star}, \dots, u_t, u_t^{\star}}_{2n}),
\end{equation*} 
be the alternating star cumulant of $u_{2t}$ of even length $2n$ and recall from \cite{DGPN}, Theorem 5.1, that:
\begin{equation}\label{RR1}
-\frac{1}{n}\frac{d}{dt}g_n(t) = g_n(t) + \sum_{m=1}^{n-1}g_m(t)g_{n-m}(t), \quad n \geq 2,
\end{equation}
with the initial value $g_1(t) = 1-e^{-t}$. Equivalently, the generating series 
\begin{equation*}
g(t,z) = \frac{1}{2} + \sum_{n \geq 1}g_n(t)z^n
\end{equation*}
converging in a neighborhood of $z = 0$ satisfies the following non linear pde: 
\begin{equation*}
\partial_t g + 2zg\partial_zg = z, \quad g(0,z) = 1/2.  
\end{equation*}
Using the method of characteristics, the following non linear equation was obtained in \cite{DGPN}, Theorem 5.4:
\begin{equation*}
[g(t,\chi_t(z))]^2 = \chi_t(z) +\frac{z^2}{4}
\end{equation*}
in a neighborhood of $z=1$, where 
\begin{equation*}
\chi_t(z) := \frac{z^2(1-z^2)e^{tz}}{[(1+z) - (1-z)e^{tz}]^2}.
\end{equation*}
Since $\chi_t'(1) \neq 0$ then $\chi_t$ is locally invertible therefore,  
\begin{equation}\label{Charac}
[g(t,z)]^2 = z +\frac{[\chi_t^{-1}(z)]^2}{4}
\end{equation}
in a neighborhood of $z=0$. As a matter of fact, it suffices and remains to compute the Taylor coefficients of the local inverse $\chi_t^{-1}$ of $\chi_t$ (for fixed $t > 0$) in order to get the explicit expression of $g_n(t)$. To proceed, we use the Lagrange inversion formula and prove the following: 
\begin{pro}\label{Pro1}
For any $n \geq 1$, there exist polynomials $(P_k^{(n)})_{k \geq 0}$ such that the $n$-th Taylor coefficient $a_n(t)$ of $\chi_t^{-1}$ reads 
\begin{align*}
a_n(t) = 2\frac{(-1)^n}{n}\sum_{k=1}^{n}\binom{2n}{n-k}e^{-kt}P_{k-1}^{(n)}(t).
\end{align*}
Moreover, 
\begin{align*}
\frac{e^{-kt}}{n}P_{k-1}^{(n)}(t) = 2 \frac{kt^{2n}}{(2n)!} \int_0^{\infty} x^{2n-1} \phi(u_{2(t+x)}^k) dx. 
\end{align*}
\end{pro}

\begin{proof}
According to Lagrange inversion formula, the $n$-th Taylor coefficient of $\chi_t^{-1}$ is given by 
\begin{equation*}
a_n(t) = \frac{1}{n!}\partial_z^{n-1}\left[\frac{z-1}{\chi_t(z)}\right]^n_{z=1}, \quad n \geq 1. 
\end{equation*}
Set 
\begin{equation*}
\xi_t(z) := \frac{z-1}{z+1}e^{tz/2},
\end{equation*} 
then 
\begin{equation*}
\chi_t(z) = -\frac{z^2}{4}\alpha^{-1}(\xi_{2t}(z))
\end{equation*}
so that we are led to the expansion of
\begin{equation*}
z \mapsto (-1)^n\frac{(z-1)^n}{z^{2n}} \frac{(1+\xi_{2t}(z))^{2n}}{[\xi_{2t}(z)]^n}
\end{equation*}
around $z=1$. To this end, we first use the generalized binomial theorem:
\begin{equation}\label{For1}
\frac{1}{z^{2n}} = \sum_{k \geq 0}(-1)^{k}\frac{(2n)_k}{k!}(z-1)^k.
\end{equation}
Next comes the expansion of 
\begin{equation*}
(z-1)^n\frac{4^n}{[\alpha^{-1}(z)]^n},
\end{equation*}
which can be read off from the proof of Proposition 3.1. in \cite{Demni}: 
\begin{equation}\label{For2}
\sum_{m \geq 0} \left\{\sum_{k = 0}^{m \wedge 2n} \binom{2n}{k}e^{(k-n)t}L_{m-k}^{(n-m)}(2(n-k)t)\right\}\frac{(z-1)^{m}}{2^{m-n}}.
\end{equation}
Gathering \eqref{For1} and \eqref{For2}, we get
\begin{align*}
a_n(t) = &\frac{(-1)^n}{n}\sum_{m=0}^{n-1}\frac{(-1)^{n-1-m}}{2^{m-n}}\frac{(2n)_{n-1-m}}{(n-1-m)!}\sum_{k = 0}^{m \wedge 2n} \binom{2n}{k}e^{(k-n)t}L_{m-k}^{(n-m)}(2(n-k)t)
\\& = -\frac{2^n}{n}\sum_{m=0}^{n-1}\frac{(-1)^{m}}{2^m}\frac{(2n)_{n-1-m}}{(n-1-m)!}\sum_{k = 0}^{m} \binom{2n}{k}e^{(k-n)t}L_{m-k}^{(n-m)}(2(n-k)t)
\\& =  -\frac{2^n}{n}\sum_{k=0}^{n-1}\binom{2n}{k}e^{-(n-k)t}\sum_{m = k}^{n-1}\frac{(-1)^{m}}{2^m}\frac{(2n)_{n-1-m}}{(n-1-m)!}L_{m-k}^{(n-m)}(2(n-k)t).
\end{align*}
Performing the index change $k \mapsto n-k$  followed by $m \mapsto n-1-m$, we end up with
\begin{align*}
a_n(t) &=  2\frac{(-1)^n}{n}\sum_{k=1}^{n}\binom{2n}{n-k}e^{-kt}\sum_{m = 0}^{k-1}(-2)^{m}\frac{(2n)_{m}}{m!}L_{k-1-m}^{(m+1)}(2kt)
\\& := 2\frac{(-1)^n}{n}\sum_{k=1}^{n}\binom{2n}{n-k}e^{-kt}P_{k-1}^{(n)}(t)
\end{align*}
where 
\begin{equation*}
P_{k-1}^{(n)}(t) := \sum_{m = 0}^{k-1}(-2)^{m}\frac{(2n)_{m}}{m!}L_{k-1-m}^{(m+1)}(2kt). 
\end{equation*}
Finally, the Gamma integral 
\begin{equation*}
(2n)_m = \frac{1}{\Gamma(2n)} \int_0^{\infty} e^{-x}x^{2n+m-1}dx, \quad n \geq 1, 
\end{equation*}
and the differentiation rule \eqref{DiffRule} lead to
\begin{align*}
P_{k-1}^{(n)}(t) &= \frac{1}{\Gamma(2n)} \int_0^{\infty} e^{-x}x^{2n-1} \sum_{m = 0}^{k-1}\frac{(2x)^m}{m!}(-1)^mL_{k-1-m}^{(m+1)}(2kt) dx  
\\& = \frac{1}{\Gamma(2n)} \int_0^{\infty} e^{-x}x^{2n-1} \sum_{m = 0}^{k-1}\frac{(2x)^m}{m!}\frac{d^mL_{k-1}^{(1)}}{d^mu}(2kt) dx
\\& = \frac{1}{\Gamma(2n)} \int_0^{\infty} e^{-x}x^{2n-1}L_{k-1}^{(1)}(2kt+2x) dx.
\end{align*}
As a result, for any $k, n \geq 1$,
\begin{align*}
\frac{e^{-kt}}{n}P_{k-1}^{(n)}(t) &= \frac{2}{(2n)!} \int_0^{\infty} e^{-(kt+x)}x^{2n+m-1} L_{k-1}^{(1)}(2kt+2x) dx 
\\& = 2\frac{t^{2n}}{(2n)!} \int_0^{\infty} x^{2n-1} e^{-k(t+x)}L_{k-1}^{(1)}(2k(t+x)) dx 
\\& = 2\frac{kt^{2n}}{(2n)!} \int_0^{\infty} x^{2n-1} \phi(u_{2(t+x)}^k) dx
\end{align*}
as desired. 
\end{proof}

\begin{remark}
Equating the Taylor coefficients of both sides of \eqref{Charac} yields:
\begin{equation}\label{EAC}
- \frac{1}{n}\frac{d}{dt}g_n(t) = \frac{1}{4}\left[2a_n(t) + \sum_{m=1}^{n-1}a_m(t)a_{n-m}(t)\right], \quad n \geq 2,
 \end{equation}
 which holds true for $n=1$ if the sum in the right hand side is considered as empty. Thus, \eqref{EAC} provides an explicit, yet complicated, expression of $g_n(t)$. 
\end{remark}

As to the alternating star cumulants of odd length $2n+1$:
\begin{equation*}
h_n(t) := \kappa_{2n+1}(\underbrace{u_t,u_t^{\star}, \dots, u_t, u_t^{\star}}_{2n}, u_t), \quad n \geq 1,
\end{equation*}
they can be derived inductively from $(g_n(t))_{n \geq 1}$ as follows.
\begin{pro}
Set $h_0(t) := \kappa_1(u_t) = e^{-t/2}$. Then, for any $n \geq1$, 
\begin{equation}\label{RR2}
\sum_{j=0}^{n-1}h_j(t) \cdot h_{n-j-1}(t) = \frac{1}{n} \frac{d}{dt}g_n(t).
\end{equation}
\end{pro}
\begin{proof}
The proposition holds true for $n=1$ by direct computations since $g_1(t) = 1-e^{-t}$. Now, let $m \geq 2$ and take $k = m-1$  and 
\begin{equation*}
\sigma = \{1\}, \dots, \{n-2\}, \{m-1,m\}
\end{equation*}
in \eqref{PE}. Then, the proof of Lemma 3.4 in \cite{DGPN} shows that the partitions $\pi \in NC(m)$ satisfying $\pi \vee \sigma = 1_m$ are exactly $1_m$ and those having only two blocks $V_1, V_2$ which in addition separate $m$ and $m-1$. Consequently,  
\begin{multline*}
\kappa_{m-1} (a_1, \ldots, a_{m-2}, a_{m-1} \cdot a_m) = \kappa_m (a_1, \ldots, a_{m-2}, a_{m-1} , a_m)
\\ + \sum_{\substack{\pi = \{ V_1, V_2 \} \in NC(m) \\  m \in V_1, m-1 \in V_2}}  \kappa_{|V_1|} ((a_i)_{i \in V_1})\cdot \kappa_{|V_2|} ((a_i)_{i \in V_2}).
\end{multline*}
But if $\pi \neq 1_m$ then either $V_1 = \{m\}$ and $V_2 = \{1, \dots, m-1\}$ or $V_2$ is nested inside $V_1$ since $\pi$ is noncrossing. In the latter case, 
\begin{equation*}
V_2 = \{j,\dots, m-1\}, \quad V_1 = \{1,\dots, j-1\} \cup \{m\}, \quad j \geq 2. 
\end{equation*}
In particular, if $m= 2n, n \geq 2$ and 
\begin{equation*}
a_{2i+1} = u_t, \quad 0 \leq i \leq n-1, \quad a_{2i} = u_t^{\star}, \quad 0 \leq i \leq n,
\end{equation*}
and since  
\begin{equation*}
\kappa_{2n-1} (u_t, u_t^{\star}, \dots, u_t^{\star}, {\bf 1}) = 0 
\end{equation*}
by freeness of $u_tu_t^{\star} = {\bf 1}$ with $\{u_t, u_t^{\star}\}$, then
\begin{align*}
- g_n(t) & = 2\kappa_1(u_t) h_{n-1}(t) + \sum_{j=1}^{n-1}g_j(t) \cdot g_{n-j}(t) + \sum_{j=1}^{n-2}h_j(t) \cdot h_{n-j-1}(t)
\\& = 2e^{-t/2} h_{n-1}(t) + \sum_{j=1}^{n-1}g_j(t) \cdot g_{n-j}(t) + \sum_{j=1}^{n-2}h_j(t) \cdot h_{n-j-1}(t).
\end{align*}
Using \eqref{RR1}, we are done.
\end{proof}

\section{R-transform of the free Jacobi process} 
Recall the definition of the free Jacobi process associated with a single projection $P$: 
\begin{equation*}
J_t = Pu_tPu_t^{\star}P, \quad t \geq 0,
\end{equation*}
and assume $\phi(P) = 1/2$. Then the description supplied in \eqref{DescJac}  of the spectral distribution of $J_t$ at a fixed time $t > 0$ shows that the free cumulants of $J_t$ in the compressed algebra coincide with 
\begin{equation*}
\kappa_n\left(\frac{u_{2t} + u_{2t}^{\star} + 2{\bf 1}}{4}\right), \quad n \geq 1.
\end{equation*}
By multi-linearity of the free cumulant functional and the freeness of ${\bf 1}$ with $\{u_{2t}, u_{2t}^{\star}$ (\cite{NS}), the latter may be written as a sum of star cumulants:
\begin{equation*}
\kappa_n\left(\frac{u_{2t} + u_{2t}^{\star} + 2{\bf 1}}{4}\right) = \frac{1}{4^n} \sum_{\epsilon_1, \dots, \epsilon_n \in \{1, \star\}} \kappa_n(u_t^{\epsilon_1}, \dots, u_t^{\epsilon_n}).
\end{equation*}
which, up to our best knowledge, are only known explicitly in the few cases dealt with in \cite{DGPN}. Nonetheless, we may seek an expression of the $R$-transform of $J_t$ relying on the functional equation \eqref{FuncEq}. To this end, recall from \cite{DHH} the moment generating function of $J_t$:
\begin{equation}\label{MomGenJac}
M_t(z) := \frac{1}{\phi(P)} \sum_{n \geq 0}\phi(J_t^n)z^n =  \frac{1}{\sqrt{1-z}}[1+2U_{\nu_{2t}}(\alpha(z))], \quad |z| < 1,
\end{equation}
where
\begin{equation}\label{MGFU}
U_{\nu_t}(z) := \sum_{k \geq 1}\phi(u_t^k) z^k  = \sum_{k \geq 1}\frac{e^{-kt/2}}{k}L_{k-1}^{(1)}(kt) z^k,
\end{equation}
Set $w_t(z) := zM_t(z)$, 
\begin{equation*}
R_t(z) := \sum_{n \geq 1}\kappa_n(J_t)z^n,
\end{equation*}
and 
\begin{equation*}
F_t(z) := \frac{z}{1+R_t(z)} =: \sum_{n \geq 1} b_n(t)z^n.
\end{equation*}
Then
\begin{pro}
For any $n \geq 1$, 
\begin{align*} 
b_n(t) = \frac{1}{n}\sum_{j=0}^{n-1} \frac{((1-n)/2)_{n-1-j}}{4^j(n-1-j)!} \sum_{k=0}^{j}Q_{k}^{(n)}(t) \binom{2j}{j-k}e^{-kt},
\end{align*}
where $Q_0^{(n)}(t) = 1$ and 
\begin{align}\label{IntRep1}
\frac{e^{-kt}}{n}Q_{k}^{(n)}(t) = (-2)\frac{t^{n+1}}{n!}\int_0^{\infty}x^{n} \phi(u_{2(t+x)}^k)dx, \quad k \geq 1. 
\end{align}
\end{pro}
\begin{proof}
From \eqref{FuncEq} and Lagrange inversion formula, it is readily seen that
\begin{equation*}
b_n(t) = \frac{1}{n!} \partial_z^{n-1} \frac{ (1-z)^{n/2}}{[1+2U_{2t}(\alpha(z))]^n}_{|z=0}.  
\end{equation*}
Since $U_{2t}(0) = 0$, then 
\begin{align*}
\frac{1}{[1+2U_{\nu_{2t}}(z)]^n} & = 1+ \sum_{m \geq 1}\frac{(n)_m}{m!}(-2)^m[U_{\nu_{2t}}(z)]^m 
\end{align*}
for small enough $|z|$. 
Therefore we need to expand $[U_{\nu_{2t}}]^m, m \geq 1$ and this task is achieved in the following lemma:
\begin{lem}\label{Lem1}
For any $m \geq 1$ and any complex number $|z| < 1$, 
\begin{equation*}
[U_{\nu_{2t}}]^{m}(z) =  m\sum_{j \geq m}L_{j-m}^{(m)}(2jt)\frac{(e^{-t}z)^j}{j}. 
\end{equation*}
\end{lem}
\begin{proof}
This expansion is an instance of formula 1.3 in \cite{Coh} (see also \cite{Man-Sri}, p.378): substitute there 
\begin{equation*}
b = 0, \quad v= m+1, \quad a= \frac{1}{m+1}, \quad x=2(m+1)t,
\end{equation*}
to get   
\begin{equation}\label{SerLag}
\sum_{j \geq 0} \frac{1}{j+m} L_j^{(m)}(2(j+m)t)(e^{-t}z)^{j+m} = \frac{1}{m}\frac{(e^{-t}z)^{m}}{(1-u)^{m}}e^{2tmu/(u-1)}.
\end{equation}
In the right hand side of the last equality, $u = u(t,z)$ belongs to the open unit disc $\mathbb{D}$ and is implicitly defined by
\begin{equation*}
e^{-t}z = ue^{2tu/(1-u)} \quad \Leftrightarrow \quad z = ue^{t(1+u)/(1-u)}.
\end{equation*}
Set
\begin{equation*}
Z := \frac{u+1}{1-u},
\end{equation*}
then straightforward computations show that $\xi_{2t}(Z) = z, \Re(Z) \geq 0$. But $\xi_{2t}$ is a one-to-one map from the Jordan domain 
\begin{equation*}
\Gamma_{2t} := \{\Re(Z) > 0, \xi_{2t}(Z) \in \mathbb{D}\}
\end{equation*}
onto $\mathbb{D}$ whose compositional inverse is $1+2U_{\nu_{2t}}$ (\cite{Biane}, Lemma 12). Hence, $Z = 1+2U_{\nu_{2t}}(z)$ so that
\begin{equation*}
u = \frac{Z-1}{Z+1}
\end{equation*}
is uniquely determined in the open unit disc. Substituting 
\begin{eqnarray*}
u = (ze^{-t})e^{2ut/(u-1)} &=& \frac{U_{\nu_{2t}}(z)}{1+U_{\nu_{2t}}(z)} \\ 
\frac{1}{1-u} &=& 1+U_{\nu_{2t}}(z)
\end{eqnarray*}
in the right-hand side of \eqref{SerLag}, the lemma is proved. 
\end{proof}

From this lemma, It follows that 
\begin{align*}
\frac{1}{[1+2U_{\nu_{2t}}(z)]^n} 
 = 1+ \sum_{j \geq 1} \frac{e^{-jt}}{j}z^j \sum_{m =1}^j \frac{(n)_m}{(m-1)!}(-2)^mL_{j-m}^{(m)}(2jt)
\end{align*}
for small enough $|z|$. Set $Q_0^{(n)}(t) = 1$ and 
\begin{align*}
Q_j^{(n)}(t) &:= \frac{1}{j}\sum_{m =1}^j\frac{(n)_m}{(m-1)!}(-2)^mL_{j-m}^{(m)}(2jt), \quad j \geq 1.
\end{align*}
Then Brown's Theorem yields
\begin{equation*}
\frac{1}{[1+2U_{\nu_{2t}}(\alpha(z))]^n} = \frac{1}{\sqrt{1-z}}\sum_{m \geq 0} \left\{\sum_{j=0}^m\binom{2m}{m-j}e^{-jt}Q_j^{(n)}(t)\right\}\frac{z^m}{4^m}.
\end{equation*}
On the other hand, the generalized binomial Theorem yields
\begin{equation*}
(1-z)^{(n-1)/2} = \sum_{m \geq 0} \frac{((1-n)/2)_m}{m!} z^m, \quad |z| < 1,
\end{equation*}
therefore 
\begin{align*}
\frac{ (1-z)^{n/2}}{[1+2U_{\nu_{2t}}(\alpha(z))]^n} & = \sum_{m \geq 0} \left\{\sum_{j=0}^m \frac{((1-n)/2)_j}{4^{m-j}j!}\sum_{k=0}^{m-j}\binom{2m-2j}{m-j-k}e^{-kt}Q_k^{(n)}(t)\right\}z^m.
\end{align*}
Extracting the $(n-1)$-th term of this series, we get
\begin{align*}
b_n(t) & = \frac{1}{n}\sum_{j=0}^{n-1} \frac{((1-n)/2)_j}{4^{n-1-j}j!} \sum_{k=0}^{n-j-1}e^{-kt}Q_{k}^{(n)}(t) \binom{2n-2j - 2}{n-j-k-1}.
\end{align*}
Performing the index change $k \mapsto n-k-j-1$ for fixed $j$ followed by $j \mapsto n-1-j$ and $k \mapsto j-k$, we end up with
\begin{align*} 
b_n(t) & = \frac{1}{n}\sum_{j=0}^{n-1} \frac{((1-n)/2)_j}{4^{n-j-1}j!} \sum_{k=0}^{n-1-j}Q_{n-1-j-k}^{(n)}(t) \binom{2n-2j - 2}{k}e^{-(n-j-1-k)t}
\\& = \frac{1}{n}\sum_{j=0}^{n-1} \frac{((1-n)/2)_{n-j-1}}{4^j (n-1-j)!} \sum_{k=0}^{j}Q_{j-k}^{(n)}(t) \binom{2j}{k}e^{-(j-k)t}
\\& = \frac{1}{n}\sum_{j=0}^{n-1} \frac{((1-n)/2)_{n-1-j}}{4^j(n-1-j)!} \sum_{k=0}^{j}Q_{k}^{(n)}(t) \binom{2j}{j-k}e^{-kt}.
\end{align*}
Finally, the integral representation \eqref{IntRep1} follows from the same lines written at the end of the proof of proposition \ref{Pro1}: 
\begin{align*}
\frac{e^{-kt}}{n}Q_k^{(n)}(t) & =  -e^{-kt}\frac{2}{kn}\sum_{m =0}^{k-1}\frac{(n)_{m+1}}{m!}(-2)^mL_{k-m-1}^{(m+1)}(2kt)
\\& = -\frac{2}{kn!} \int_0^{\infty}x^ne^{-(kt+x)} \sum_{m =0}^{k-1}\frac{(-2x)^m}{m!}L_{k-m-1}^{(m+1)}(2kt) dx
\\& = -2\frac{t^{n+1}}{kn!} \int_0^{\infty}x^ne^{-k(t+x)}L_{k-1}^{(1)}(2k(t+x)) dx
\\& = -2 \frac{t^{n+1}}{n!} \int_0^{\infty}x^n \phi(u_{2(t+x)}^k) dx. 
\end{align*}
\end{proof}


\begin{remark}
As pointed out to the author by C. Dunkl, $b_n(t)$ can be expressed through the ${}_3F_2$-hypergeometric function (\cite{AAR}). To see this, interchange the summation order in the expression of $b_n(t)$: 
\begin{align*}
b_{n}(t) =\frac{1}{n}\sum_{k=0}^{n-1}Q_{k}^{(n)}(t) e^{-kt}\sum_{j=k}^{n-1}\frac{\left(\left(1-n\right)/2\right) _{n-1-j}}{4^{j}\left(n-1-j\right)!}\binom{2j}{j-k},
\end{align*}
then use the Legendre duplication formula (\cite{AAR}):
\begin{equation*}
\frac{1}{4^{j}}\binom{2j}{j-k}=\frac{1}{4^{j}}\frac{\left(  2j\right)!}{\left(  j-k\right)  !\left(  j+k\right)  !}
= \frac{j!\left(1/2\right)  _{j}}{\left(  j-k\right)  !\left(j+k\right)  !},
\end{equation*}
to get
\begin{align*}
\sum_{j=k}^{n-1}\frac{\left(\left(1-n\right)/2\right) _{n-1-j}}{4^{j}\left(n-1-j\right)!}\binom{2j}{j-k} & = \sum_{i=0}^{n-1-k}\frac{\left(  \left(  1-n\right)  /2\right)  _{i}}{i!}\frac{\left(  n-1-i\right)  !\left(1/2\right)_{n-1-i}}{\left(n-1-k-i\right)  !\left(  n-1+k-i\right)!}.
\end{align*}
Now, the relations 
\begin{equation*}
\left(a\right) _{n-1-i}=\left(-1\right)^{i}\frac{\left(a\right)_{n-1}}{\left(2-a-n\right)  _{i}}, \quad \left(N-i\right)!=\left(-1\right)^{i}\frac{N!}{\left(-N\right)  _{i}},
\end{equation*}
follow from the definition of the Pochhammer symbol. Consequently,
\begin{align*}
\sum_{j=k}^{n-1}\frac{\left(\left(1-n\right)/2\right) _{n-1-j}}{4^{j}\left(n-1-j\right)!}\binom{2j}{j-k} & = 
\frac{\left(n-1\right)!\left(1/2\right)_{n-1}}{\left(n-1-k\right)  !\left(  n-1+k\right)  !}\sum_{i=0}^{n-1-k}\frac{\left(\left(1-n\right)  /2\right)  _{i}}{i!}\frac{\left(1+k-n\right)  _{i}\left(1-k-n\right)  _{i}}{\left(  1-n\right)  _{i}\left((3/2)-n\right)_{i}}
\\& = \frac{1}{4^{n-1}}\binom{2n-2}{n-1-k}~_{3}F_{2}\left(\genfrac{}{}{0pt}{}{1-k-n,1+k-n, (1-n)/2}{1-n, (3/2)-n};1\right).
\end{align*}
In particular,
\begin{align*}
\sum_{j=0}^{n-1}\frac{\left(\left(1-n\right)/2\right) _{n-1-j}}{4^{j}\left(n-1-j\right)!}\binom{2j}{j} = \frac{1}{4^{n-1}}\binom{2n-2}{n-1}~_{2}F_{1}\left(\genfrac{}{}{0pt}{}{1-n, (1-n)/2}{(3/2)-n};1\right).
\end{align*}
\end{remark}

\begin{remark}
Set 
\begin{equation*}
s_k(t) := e^{kt}\phi(u_t^k) = \frac{1}{k}L_{k-1}^{(1)}(kt), \quad k \geq 1,
\end{equation*}
then Lemma \ref{Lem1} asserts that for any $j \geq m \geq 1$, 
\begin{equation}\label{DiffEq}
\sum_{j_1+\dots+j_m = j} s_{j_1}(t)\dots s_{j_m}(t) = \frac{1}{j}L_{j-m}^{(m)}(jt). 
\end{equation}
When $m=2$, this identity is equivalent to the known fact (see e.g. \cite{Rai})
\begin{equation}\label{DiffEq1}
-\frac{1}{k}\partial_ts_k(t) = \sum_{j=1}^{k-1}s_j(t)s_{j-k}(t), \, k \geq 2, \,\, s_1(t) = 1,
\end{equation}
since 
\begin{equation*}
-\frac{1}{k}\partial_ts_k(t)  = \frac{2}{k}L_{k-2}^{(2)}(2kt).  
\end{equation*}
More generally, \eqref{DiffEq} may be derived inductively from \eqref{DiffEq1} after differentiation with respect to $t$ and multiple index changes.  
\end{remark}

\section{S-transform of the free Jacobi process}
In the same way the $R$-transform linearizes the free additive convolution, the $S$-transform does so for the free multiplicative convolution of probability distributions on the unit circle and on the positive real line.  
According to \eqref{FuncEq1}, we need to compute the compositional inverse of $M_t - 1$ around $z=0$ which exists since $\phi(J_t) \neq 0$. To this end, we first need a lemma. 
\begin{lem}
Let $z$ be a complex number with $|z| < 1$ and $\Re(1+2z) > 0$. Then, for any $n \geq 1$ and $t > 0$, 
\begin{equation*}
G_{n,t}(z) := \frac{1}{(1+z)^n(1+e^{-t(1+2z)})^n} = \sum_{m \geq 0}H_m^{(n)}(t)z^m
\end{equation*}
where 
\begin{equation*}
H_m^{(n)}(t) = \frac{(-1)^m}{m!}\sum_{k=0}^m \binom{m}{k}(n)_{m-k}(2t)^k \frac{d^k}{dt^k} \frac{1}{(1+e^{-t})^n}.
\end{equation*}
\end{lem}
\begin{proof}
Since $\Re(1+2z) > 0$ then $|e^{-t(1+2z)}| < 1$ so that 
\begin{equation*}
\frac{1}{(1+e^{-t(1+2z)})^n} = \sum_{j \geq 0}\frac{(n)_j}{j!}(-1)^je^{-tj(1+2z)}.
\end{equation*}
But then, \eqref{GFC} and \eqref{CL} entail 
\begin{equation*}
\frac{1}{(1+z)^n}e^{-2tjz} = \sum_{m \geq 0}C_m(-n, 2tj)\frac{(-2tjz)^m}{m!} = \sum_{m \geq 0} L_m^{-(n+m)}(2jt)z^m.
\end{equation*}
Consequently,
\begin{equation*}
H_m^{(n)}(t) = \sum_{j \geq 0}\frac{(n)_j}{j!}(-e^{-t})^jL_m^{-(n+m)}(2jt).
\end{equation*}
In order to get the expression displayed in the lemma, we use \eqref{DefL} to write 
\begin{align*}
L_m^{-(n+m)}(2jt) & = \frac{1}{m!}\sum_{k=0}^m (-1)^k\binom{m}{k}(-n-m+k+1)_{m-k}(2jt)^k
\\& = \frac{(-1)^m}{m!}\sum_{k=0}^m \binom{m}{k} (n)_{m-k}(2jt)^k
\end{align*}
whence
\begin{align*}
H_m^{(n)}(t) &= \frac{(-1)^m}{m!}\sum_{k=0}^m \binom{m}{k}(n)_{m-k}(2t)^k \sum_{j \geq 0} \frac{(n)_j}{j!} (-e^{-t})^j j^k
\\& = \frac{(-1)^m}{m!}\sum_{k=0}^m \binom{m}{k}(n)_{m-k}(2t)^k  \frac{d^k}{dt^k} \frac{1}{(1+e^{-t})^n}. 
\end{align*}
\end{proof}

With the help of this lemma, we derive the following: 
\begin{pro}
Set 
\begin{eqnarray*}
V_0^{(n)}(t) &:=& 2^n H_0^{(n)}(t) = \frac{2^n}{(1+e^{-t})^n}, \\
V_j^{(n)}(t) &:=& 2^n \frac{(e^{-t}z)^j}{j} \sum_{m=1}^j mH_m^{(n)}(t) L_{j-m}^{(m)}(2jt), \quad j \geq 1,\\
d_j^{(n)} &: = & \sum_{k=0}^j \frac{(-n-1)_k}{k!}\frac{(2n+1)_{j-k}}{(j-k)!}, j \geq 0.
\end{eqnarray*}
Then the inverse function $(M_t -1)^{-1}$ admits the Taylor expansion 
\begin{equation*}
(M_t -1)^{-1}(z) := \sum_{n \geq 1}c_n(t) z^n
\end{equation*}
near the origin, where
\begin{equation*}
c_n(t) = \frac{1}{n4^{n-1}} \sum_{j=0}^{n-1} \binom{2n-2}{n-1-j} \sum_{k=0}^j d_{j-k}^{(n)}V_j^{(n)}(t).
\end{equation*}
Consequently, the $S$-transform of $\mu_t$, say $S_t, t > 0$, reads
\begin{equation*}
S_t(z) = c_1(t) + \sum_{n \geq 1}[c_{n+1}(t) + c_n(t)] z^n.
\end{equation*}
\end{pro}
\begin{proof} 
From \eqref{MomGenJac}, we can rewrite $M_t(z)$ as
\begin{equation*}
M_t(z) = \frac{1+\alpha(z)}{1-\alpha(z)} [1+U_{\nu_{2t}}(\alpha(z))].
\end{equation*}
It follows that 
\begin{align*}
c_n(t) &= \frac{1}{n!}\partial_z^{n-1}\left\{\frac{z(1-\alpha(z))}{(1+\alpha(z))(1+2U_{\nu_{2t}}(z)) -(1-\alpha(z))}\right\}^n{}_{|z=0}
\\& = \frac{1}{n!}\partial_z^{n-1}\left\{\frac{4\alpha(z)}{(1+\alpha(z))^2}\frac{(1-\alpha(z))}{(1+\alpha(z))(1+2U_{\nu_{2t}}(\alpha(z))) -(1-\alpha(z))}\right\}^n{}_{|z=0}.
\end{align*}
Moreover, by the virtue of Brown's Theorem, it suffices to expand 
\begin{align*}
\frac{(1-\alpha(z))^{n+1}}{(1+\alpha(z))^{2n+1}} \left\{\frac{4\alpha(z)}{(1+\alpha(z))(1+2U_{\nu_{2t}}(\alpha(z))) -(1-\alpha(z))}\right\}^n,
\end{align*}
and to this end, we consider 
\begin{equation*}
\frac{4z}{(1+z)(1+2U_{2t}(z)) -(1-z)}.
\end{equation*}
From the inverse relation $\xi_{2t}(1+2U_{\nu_{2t}}(z)) = z, |z| < 1$, it follows that 
\begin{align*}
U_{\nu_{2t}}(z) = z(1+U_{\nu_{2t}}(z))e^{-(1+2tU_{\nu_{2t}}(z))}
\end{align*}
whence
\begin{align*}
\frac{4z}{(1+z)(1+2U_{\nu_{2t}}(z)) -(1-z)} &= \frac{2z}{z +(1+z)U_{\nu_{2t}}(z)} 
\\ &= \frac{2}{1 +(1+z)(1+U_{\nu_{2t}}(z))e^{-t(1+2U_{\nu_{2t}}(z))}} 
\\& = \frac{2}{[1+U_{\nu_{2t}}(z)][1+ e^{-t(1+2U_{\nu_{2t}}(z))}]}.
\end{align*}
Now, the previous lemma and the fact that $\Re(1+2U_{\nu_{2t}}(z)) > 0$ for $|z| < 1$ (\cite{Biane}) show that for any $n \geq 1$ and small $|z|$,  
\begin{equation*}
\frac{1}{[1+U_{\nu_{2t}}(z)]^n[1+ e^{-t(1+2U_{\nu_{2t}}(z))}]^n} = G_{n,t}(U_{\nu_{2t}}(z)) = \sum_{m \geq 0}H_m^{(n)}(t)[U_{\nu_{2t}}(z)]^m,
\end{equation*}
which by Lemma \ref{Lem1} is further expanded as 
\begin{align*}
\frac{1}{[1+U_{\nu_{2t}}(z)]^n[1+ e^{-t(1+2U_{\nu_{2t}}(z))}]^n}  & = H_0^{(n)}(t) + \sum_{m \geq 1}mH_m^{(n)}(t)\sum_{j \geq m}L_{j-m}^{(m)}(2jt)\frac{(e^{-t}z)^j}{j}
\\& = \frac{1}{(1+e^{-t})^n} + \sum_{j \geq 1}\frac{(e^{-t}z)^j}{j}\sum_{m=1}^j mH_m^{(n)}(t)L_{j-m}^{(m)}(2jt)
\\& = \frac{1}{2^n}\sum_{j \geq 0}V_j^{(n)}(t) z^j.
\end{align*}
Moreover, 
\begin{equation*}
\frac{(1-z)^{n+1}}{(1+z)^{2n+1}} = \sum_{j \geq 0} \sum_{k=0}^j \frac{(-n-1)_k}{k!}\frac{(2n+1)_{j-k}}{(j-k)}!z^j = \sum_{j \geq 0} d_j^{(n)}z^j
\end{equation*}
so that 
\begin{equation*}
\frac{(1-z)^{n+1}}{(1+z)^{2n+1}}\left[\frac{4z}{(1+z)(1+2U_{2t}(z)) -(1-z)}\right]^n  = \sum_{j \geq 0}\sum_{k=0}^j d_{j-k}^{(n)}V_j^{(n)}(t)z^j.
\end{equation*}
Composing with the map $\alpha$, we get from Brown's Theorem
\begin{multline*}
\frac{(1-\alpha(z))^{n+1}}{(1+\alpha(z))^{2n+1}} \left\{\frac{4\alpha(z)}{(1+\alpha(z))(1+2U_{\nu_{2t}}(\alpha(z))) -(1-\alpha(z))}\right\}^n = \frac{1-\alpha(z)}{1+\alpha(z)}
\\ \sum_{m\geq 0}\frac{z^m}{4^m} \sum_{j=0}^m \binom{2m}{m-j} \sum_{k=0}^j d_{j-k}^{(n)}V_j^{(n)}(t),
\end{multline*}
or equivalently,
\begin{equation*}
\left\{\frac{z(1-\alpha(z))}{(1+\alpha(z))(1+2U_{\nu_{2t}}(z)) -(1-\alpha(z))}\right\}^n = \sum_{m\geq 0}\frac{z^m}{4^m} \sum_{j=0}^m \binom{2m}{m-j} \sum_{k=0}^j d_{j-k}^{(n)}V_j^{(n)}(t).
\end{equation*}
The expressions of $c_n(t)$ and of $S_t$ are now obvious and the proposition is proved.
\end{proof}

\section{Schur function of $\nu_t$ and its first iterate} 
Given a probability distribution $\mu$ supported in the unit circle, its Schur function $f_{\mu}$ is defined in the open unit disc by (\cite{Sim1}, p.25):
\begin{equation*}
\frac{1+zf_{\mu}(z)}{1-zf_{\mu}(z)} = \int_{\mathbb{T}} \frac{w+z}{w-z} \mu(dw) := H_{\mu}(z).
\end{equation*}
Equivalently,
\begin{equation*}
zf_{\mu}(z) = \frac{H_{\mu}(z) -1}{H_{\mu}(z) + 1}. 
\end{equation*}
Since $H_{\mu}(0) = 1$ and is analytic in $\mathbb{D}$, then $f_{\mu}$ is analytic there and the Verblunsky coefficients $(\gamma_j)_{j \geq 0}$ of $\mu$ are defined by the following continued fraction (\cite{Sim1}, p.3): 
\begin{equation*}
f_{\mu}(z) = \gamma_0 + \frac{1-|\gamma_0|^2}{\overline{\gamma_0} + \displaystyle \frac{1}{z\gamma_1+ \displaystyle \frac{z(1-|\gamma_1|^2)}{\overline{\gamma_1} + \dots}}}.
\end{equation*}
In a practical way, the Schur algorithm allows to compute them recursively from the Schur iterates $(f_{j,\mu})_{j \geq 0}$ as follows (\cite{Sim1}): 
\begin{equation*}
f_{0,\mu} := f_{\mu}, \quad zf_{j+1,\mu}(z) = \frac{f_{j,\mu}(z) - \gamma_{j}}{1-\overline{\gamma_{j}}f_{j,\mu}(z)}, \quad \gamma_j = f_{j,\mu}(0).
\end{equation*}
When $\mu = \nu_t$ is the spectral distribution of $u_t$, its Herglotz transform $H_{\mu} = H_{\nu_t}$ is the inverse function of $\xi_t$ in the open unit disc (\cite{Biane}). As a result,
\begin{pro}
For any $t > 0$, the Schur function $f_{0,\nu_t}$ and its first iterate $f_{1,\nu_t}$ admit the following expansions:
\begin{eqnarray*}
f_{0,\nu_t}(z) & = &  e^{-t/2} -te^{-t/2} \sum_{j \geq 1}\frac{e^{-jt/2}}{j} L_{j-1}^{(1)}((j+1)t)z^j \\ 
f_{1,\nu_t}(z) 
          & = & te^{-t}(1-e^t) \sum_{j \geq 0}\frac{e^{-jt/2}}{j+1} \left[\sum_{k \geq 1}ke^{-kt} L_{j}^{(1)}((j+k+1)t)\right]z^j,
\end{eqnarray*}
for $|z| < 1$.
\end{pro}
\begin{proof}
Note first that both expansions are absolutely convergent in the open unit disc due to the following estimate (\cite{AS}, 22.14.13, p.786):
\begin{equation}\label{Bound}
|L_j^{(1)}(x)| \leq (j+1)e^{x/2}, \quad x \geq 0. 
\end{equation}
Now, using the expression  
\begin{equation*}
\xi_t(z) = \frac{z-1}{z+1}e^{tz/2}, \quad z \in \Gamma_{2t},
\end{equation*}
we readily get
\begin{equation*}
zf_{0,\nu_t}(z) = \frac{H_{\nu_t}(z) -1}{H_{\nu_t}(z) + 1} = ze^{-tH_{\nu_t}(z)/2}.
\end{equation*}
But $\nu_t$ is invariant under complex conjugation $z \mapsto \overline{z}$ so that $H_{\nu_t} = 1+2U_{\nu_t}$ where we recall that $U_{\nu_t}$ was previously defined in \eqref{MGFU}. Consequently,
\begin{equation*}
f_{0,\nu_t}(z) = e^{-tH_t(z)/2} = e^{-t/2}e^{-tU_{\nu_t}} = e^{-t/2}\sum_{m \geq 0} \frac{(-t)^m}{m!}[U_{\nu_t}]^m.
\end{equation*}
From lemma \ref{Lem1}, we readily derive 
\begin{align*}
f_{0,\nu_t}(z) &= e^{-t/2}+ e^{-t/2}\sum_{m \geq 1} \frac{(-t)^m}{m!}m\sum_{j \geq m}L_{j-m}^{(m)}(jt)\frac{(e^{-t/2}z)^j}{j}
\\& = e^{-t/2} + e^{-t/2}\sum_{m \geq 0} \frac{(-t)^{m+1}}{m!}\sum_{j \geq m+1}L_{j-m-1}^{(m+1)}(jt)\frac{(e^{-t/2}z)^j}{j}  
\\& = e^{-t/2} -t e^{-t/2} \sum_{j \geq 1}\frac{e^{-jt/2}}{j} \left\{\sum_{m=0}^{j-1}\frac{(-t)^{m}}{m!}L_{j-m-1}^{(m+1)}(jt)\right\} z^j
\\& = e^{-t/2} -te^{-t/2} \sum_{j \geq 1}\frac{e^{-jt/2}}{j} L_{j-1}^{(1)}((j+1)t)z^j,
\end{align*}
where the last equality follows from Taylor's formula. As a result, $\gamma_0(t) = e^{-t/2}$ and as such, the first Schur iterate reads 
\begin{equation*}
f_{1,\nu_t}(z) = e^{-t/2}\frac{e^{-tU_{\nu_t}} - 1}{z} \sum_{k \geq 0}e^{-kt} e^{-ktU_{\nu_t}(z)} = e^{-t/2} \sum_{k \geq 0}e^{-kt} \frac{e^{-(k+1)tU_{\nu_t}(z)} - e^{-ktU_{\nu_t}}}{z}.
\end{equation*}
But similar computations as above yield 
\begin{equation*}
e^{-ktU_{\nu_t}(z)} = 1-kt \sum_{j \geq 1}\frac{e^{-jt/2}}{j} L_{j-1}^{(1)}((j+k)t)z^j, \quad k \geq 0,
\end{equation*}
whence
\begin{align*}
e^{-t/2}\frac{e^{-(k+1)tU_{\nu_t}(z)} - e^{-ktU_{\nu_t}}}{z} & = t \sum_{j \geq 1}\frac{e^{-jt/2}}{j} \left[kL_{j-1}^{(1)}((j+k)t) - (k+1)L_{j-1}^{(1)}((j+k+1)t)\right]z^{j-1}
\\& = te^{-t} \sum_{j \geq 0}\frac{e^{-jt/2}}{j+1} \left[kL_{j}^{(1)}((j+k+1)t) - (k+1)L_{j}^{(1)}((j+k+2)t)\right]z^{j}.
\end{align*}
Finally, the estimate \eqref{Bound} shows that for any $t > 0$, the double series 
\begin{align*}
\sum_{k,j \geq 0}e^{-kt} \frac{e^{-jt/2}}{j+1}(k+m-1)L_{j}^{(1)}((j+k+m)t) z^{j}, \quad m \in \{1,2\},
\end{align*}
converges absolutely for $|z| < 1$ so that Fubini Theorem applies and yields
\begin{align*}
f_{1,\nu_t}(z) &= e^{-t/2}\sum_{k \geq 0} \frac{e^{-tU_{\nu_t}} - 1}{z} \sum_{k \geq 0}e^{-kt} e^{-ktU_{\nu_t}(z)} 
\\& = te^{-t} \sum_{j \geq 0} \frac{e^{-jt/2}}{j+1} \left\{\sum_{k \geq 0}e^{-kt} kL_{j}^{(1)}((j+k+1)t) - \sum_{k \geq 0}e^{-kt} (k+1)L_{j}^{(1)}((j+k+2)t)\right\}z^{j}
\\& =  te^{-t}(1-e^{-t}) \sum_{j \geq 0} \frac{e^{-jt/2}}{j+1} \sum_{k \geq 1}ke^{-kt} L_{j}^{(1)}((j+k+1)t)z^j.
\end{align*}
\end{proof}

\begin{remark}
From the two expansions derived above, we can compute the low-orders Verblunsky coefficients of $\nu_t$:
\begin{eqnarray*}
\gamma_1(t) & = & -\frac{te^{-t}}{1-e^{-t}},  \\ 
\gamma_2(t) & = & \frac{te^{-3t/2}[3t-2 + (2-t)e^{-t}]}{2[1-2e^{-t}+(1-t^2)e^{-2t}]}. 
\end{eqnarray*}
However, we do not succeed to get a closed formula for all of them (the computation of $\gamma_3(t)$ is already tedious). On the other hand, recall that $(\gamma_j(t))_{j \geq 0}$ are connected to the Jacobi-Szego parameters $(a_n)_n, (b_n)_n$ of the spectral distribution of $u_t + u_t^{\star}$ through the inverse Geronimus relations (\cite{Sim2}, p.968). Recall also that both sequences $(a_n)_n, (b_n)_n$ encode the J-continued fraction expansion of the Cauchy-Stieltjes transform of the spectral distribution of $u_t+u_t^{\star}$ (\cite{AAR}, \cite{Sim1}). With regard to \eqref{DescJac}, they are affine transformations of the Jacobi-Szeg\"o parameters of the spectral distribution $\mu_{t/2}$ of $J_{t/2}$. 
\end{remark}


\begin{thebibliography}{9999}
\bibitem{AS} \emph{M. Abramowitz, I.A. Stegun.} Handbook of Mathematical Functions. {\it Dover, NewYork}, 1964.
\bibitem{AAR}\emph{G. E. Andrews, R. Askey, R. Roy}. Special functions. {\it Cambridge University Press}. 1999.
\bibitem{Bia}\emph{P. Biane}. Free Brownian motion, free stochastic calculus and random matrices. {\it Fields Institute Communications} {\bf 12}, Amer. Math. Soc. Providence, RI,  1997, 1-19.
\bibitem{Biane}\emph{P. Biane}. Segal-Bargmann transform, functional calculus on matrix spaces and the theory of semi-circular and circular systems. {\it J. Funct. Anal.} {\bf 144}. no. 1, 1997.  232-286.
\bibitem{Coh}\emph{M. E. Cohen}. Some classes of generating functions for the Laguerre and Hermite polynomials. {\it Math. for Computations}. {\bf 31}, 238, 1977, 511-518. 
\bibitem{Dem}\emph{N. Demni}. Free Jacobi process. {\it J. Theo. Probab}. {\bf 21}. no.1. 2008, 118-143. 
\bibitem{Demni}\emph{N. Demni}. Free Jacobi process associated with one projection: local Inverse of the flow. {\it Complex Anal. Oper. Theory.} {\bf 10}, (2016), no. 3, 527-543.
\bibitem{DHH}\emph{N. Demni, T. Hamdi, T. Hmidi}. Spectral distribution of the free Jacobi process. {\it Indiana Univ. Math. Journal.} {\bf 61}. no. 3, (2012), 1351-1368
\bibitem{DGPN} \emph{N. Demni, M. Guay-paquet, A. Nica}. Star-cumulants of the free unitary Brownian motion. {\it Adv. in Appl. Math.} {\bf 69}, (2015), 1-45.
\bibitem{Hia-Pet}\emph{F. Hiai, D. Petz}. The semicircle law, free random variables and entropy. {\it Mathematical Surveys and Monographs, 77. American Mathematical Society, Providence, RI}, 2000. x+376 pp.
\bibitem{Levy}\emph{T. L\'evy}. Schur-Weyl duality and the heat kernel measure on the unitary group. {\it Adv. Math}. {\bf 218}, 2008, no. 2, 537-575. 
\bibitem{Man-Sri}\emph{H. L. Manocha, H. M. Srivastava}. A treatise on generating functions. {\it Ellis Horwood Series: Mathematics and its Applications}. 1984.
\bibitem{NS}\emph{A. Nica, R. Speicher}. Lectures on Combinatorics of Free Probability. {\it London Mathematical Society Lecture Note Series, 335}. 2006.   
\bibitem{Rai}\emph{E. M. Rains}. Combinatorial properties of Brownian motion on the compact classical groups. {\it J. Theor. Probab.} {\bf 10}, no. 3. 1997, 659-679.
\bibitem{Sim1}\emph{B. Simon}. Orthogonal Polynomials on the Unit Circle. Part 1. Classical Theory. {\it American Mathematical Society Colloquium Publications 54, Providence, R.I.} (2005). 
\bibitem{Sim2}\emph{B. Simon}. Orthogonal Polynomials on the Unit Circle. Part 2. Spectral Theory. {\it American Mathematical Society Colloquium Publications 54, Providence, R.I.} (2005). 
\bibitem{Vo}\emph{D. Voiculescu}. Multiplication of certain noncommuting random variables. {\it J. Operator Theory}. {\bf 18}, (1987), no.2. 223-235. 
\bibitem{Voi} \emph{D. V. Voiculescu, K. J. Dykema, A. Nica.} Free random variables. A noncommutative probability approach to free products with applications to random matrices, operator algebras and harmonic analysis on free groups. {\it CRM Monograph Series, 1. American Mathematical Society, Providence, RI}, 1992. vi+70 pp.
\end{thebibliography}
\end{document}